\documentclass{amsart}
\usepackage{amsfonts, amsbsy, amsmath, amssymb}

\input prepictex.tex
\input pictex.tex
\input postpictex.tex

\newtheorem{thm}{Theorem}[section]
\newtheorem{lem}[thm]{Lemma}

\newtheorem{rmk}[thm]{Remark}
\numberwithin{equation}{section}

\newcommand{\x}{{\tt x}}

\begin{document}

\title{Determination of a Type of Permutation Binomials over Finite Fields}

\author[Xiang-dong Hou]{Xiang-dong Hou*}
\address{Department of Mathematics and Statistics,
University of South Florida, Tampa, FL 33620}
\email{xhou@usf.edu}
\thanks{* Research partially supported by NSA Grant H98230-12-1-0245.}

\author{Stephen D. Lappano}
\address{Department of Mathematics and Statistics,
University of South Florida, Tampa, FL 33620}
\email{slappano@mail.usf.edu}

\keywords{binomial, finite field, Hermite criterion, permutation polynomial}

\subjclass[2000]{11T06, 11T55}

\begin{abstract}
Let $f=a\x+\x^{3q-2}\in\Bbb F_{q^2}[\x]$, where $a\in\Bbb F_{q^2}^*$. We prove that $f$ is a permutation polynomial of $\Bbb F_{q^2}$ if and only if one of the following occurs: (i) $q=2^e$, $e$ odd, and $a^{\frac{q+1}3}$ is a primitive $3$rd root of unity. (ii) $(q,a)$ belongs to a finite set which is determined in the paper.
\end{abstract}

\maketitle

\section{Introduction}

A polynomial $f\in\Bbb F_q[\x]$ is called a {\em permutation polynomial} (PP) of $\Bbb F_q$ if it induces a permutation of $\Bbb F_q$. While permutation monomials of $\Bbb F_q$ are obvious ($a\x^n$, $a\in\Bbb F_q^*$, $\text{gcd}(n,q-1)=1$), the situation for permutation binomials is much more interesting and challenging. The reason for a binomial to be a PP can be quite nontrivial despite the simple appearance of the binomial. In \cite{Car-Wel66}, Carlitz and Wells proved that for fixed integers $e>1$ and $c> 0$, when $q$ is large enough and satisfies the conditions $e\mid q-1$ and $\text{gcd}(c,q-1)=1$, there exists $a\in\Bbb F_q^*$ such that $\x^c(\x^{\frac{q-1}e}+a)^k$ is a PP of $\Bbb F_q$ for all $k\ge 0$. (Note that when $k=1$, the PP is a binomial.) The special cases of this result with $c=k=1$ and $e=2,3$ appeared in \cite{Car62}. Carlitz and Wells' proof of the existence result relies on a bound on the Weil sum of a multiplicative character of $\Bbb F_q$ \cite{Wei41}, \cite[Theorem~5.39]{LN}. Using the Hasse-Weil bound on the number of degree one places of a function field over $\Bbb F_q$ \cite[Theorem~V.2.3]{Sti93}, Masuda and Zieve \cite{Mas-Zie09} were able to make Carlitz-Wells' existence result (with $k=1$) more precise. They proved that if $q\ge 4$ and $\frac{q-1}e>2q(\log\log q)/\log q$, then there exists $a\in\Bbb F_q^*$ such that $\x^c(\x^{\frac{q-1}e}+a)$ is a PP of $\Bbb F_q$. Moreover, they obtained an estimate for the number of $a$'s with this property.

There are also nonexistence results on permutation binomials. Niederreiter and Robinson \cite{Nie-Rob82} proved that if there is a PP of $\Bbb F_q$ of the form $\x^m+a\x$, where $m>2$ and $a\in\Bbb F_q^*$, then either $m$ is a power of $p$ ($p=\text{char}\,\Bbb F_q$) or $q<(m^2-4m+6)^2$. An improvement of this result was obtained by Turnwald \cite{Tur88}: If there is a PP of $\Bbb F_q$ of the form $\x^m+a\x^n$, where $m>n>0$ and $a\in\Bbb F_q^*$, then either $\frac mn$ is a power of $p$ or $q\le(m-2)^4+4m-4$. For permutation binomials over prime fields, the nonexistence results are stronger. Wan \cite{Wan87} proved that if there is a PP of $\Bbb F_p$ of the form $\x^m+a\x$, where $m>1$ and $a\in\Bbb F_p^*$, then $p-1\le(m-1)\text{gcd}(m-1,p-1)$. Turnwald \cite{Tur88} considered $f=\x^m+a\x^n\in\Bbb F_p[\x]$, where $m>n>0$ and $a\in\Bbb F_p^*$, and proved that $f$ is a PP of $\Bbb F_p$ implies $p<m\cdot\max(n,m-n)$. Masuda and Zieve \cite{Mas-Zie09} improved Turnwald's bound to $p-1<(m-1)\cdot\max\{n,\text{gcd}(m-n,p-1)\}$.

Let $r\ge 2$. In \cite{Car62}, Carlitz proved that the binomial $\x^{1+\frac{q-1}2}+a\x$ ($q$ odd, $a\ne0$) cannot be a PP of $\Bbb F_{q^r}$, and he raised the same question for $\x^{1+\frac{q-1}3}+a\x$ ($q\equiv 1\pmod 3$, $a\ne 0$). 
Wan \cite{Wan87, Wan94} answered Carlitz's question by showing that $\x^{1+\frac{q-1}3}+a\x$ ($q\equiv 1\pmod 3$, $a\ne 0$) cannot be a PP of $\Bbb F_{p^r}$. Kim and Lee \cite{Kim-Lee95} proved that $\x^{1+\frac{q-1}5}+a\x$ ($q\equiv 1\pmod 5$, $a\ne 0$) cannot be a PP of $\Bbb F_{q^r}$ for $p\ne 2$. 
More generally, one may consider $\x^{1+\frac{q-1}m}+a\x\in\Bbb F_q[\x]$, where $q\equiv 1\pmod m$, $m\ge 2$, $a\ne 0$. Clearly, if $m=\frac{q-1}{p^i-1}$, where $\Bbb F_{p^i}\subset\Bbb F_q$, then $\x^{1+\frac{q-1}m}+a\x=\x^{p^i}+a\x$, which is a PP of $\Bbb F_{q^r}$ if and only if $(-a)^{(q^r-1)/(p^i-1)}\ne 1$. When $1+\frac{q-1}m$ is not a power of $p$, it is not known if the binomial can be a PP of $\Bbb F_{q^r}$. 

Let $f=\x^m+a\x^n\in\Bbb F_q[\x]$, where $m>n>0$ and $a\in\Bbb F_q^*$. The conditions that make $f$ a PP of $\Bbb F_q$ are encoded in a simple set of parameters $m,n,q,a$ in a mysterious way that is not well understood on the whole. However, when $m$ and $n$ take certain particular forms, necessary and sufficient conditions for $f$ to be a PP of $\Bbb F_q$ have been found. Niederreiter and Robinson \cite{Nie-Rob82} proved that $\x^{\frac{q+1}2}+a\x\in\Bbb F_q[x]$ ($q$ odd, $a\in\Bbb F_q^*$) is a PP of $\Bbb F_q$ if and only if $a^2-1$ is a square in $\Bbb F_q^*$; also see \cite{Car62}. Akbary and Wang \cite{Akb-Wan06} considered binomials of the form $f=\x^r(1+\x^{es})$, where $e,r,s$ are positive integers such that $s\mid q-1$, $\text{gcd}(r,s)=1$, $\text{gcd}(2e,\frac{q-1}s)=1$. They found sufficient conditions for $f$ to be a PP of $\Bbb F_q$ in terms of the period of the generalized Lucas sequence. The conditions are not entirely explicit, but their special cases do give explicit classes of permutation binomials of $\Bbb F_q$. The sufficient conditions in \cite{Akb-Wan06} were later extended by Wang \cite{Wan07} to conditions that are both necessary and sufficient. Zieve \cite{Zie09} considered $f=\x^m+a\x^n\in\Bbb F_q[\x]$, where $m>n>0$ and $a\in\Bbb F_q^*$, under the assumption that $\{\eta+\frac 1\eta:\eta\in\mu_{2d}\}\subset\mu_s$, where $s=\text{gcd}(m-n,q-1)$, $d=\frac{q-1}s$ and $\mu_s=\{x\in\Bbb F_q:x^s=1\}$. In this setting, it was shown that $f$ is a PP of $\Bbb F_q$ if and only $-a\notin\mu_d$, $\text{gcd}(n,s)=1$ and $\text{gcd}(2d,m+n)\le 2$. (This characterization implies the aforementioned sufficient conditions in \cite{Akb-Wan06}; see \cite{Zie09}.) 

Recently, we determined all PPs of $\Bbb F_{q^2}$ of the form $a\x+\x^{2q-1}$ ($a\in\Bbb F_{q^2}^*$) \cite{Hou13,Hou-p}. The techniques introduced there are applicable to binomials of similar types. In the present paper, we determine all PPs of $\Bbb F_{q^2}$ of the form $a\x+\x^{3q-2}$. Our main result is the following.

\begin{thm}\label{T1.1}
Let $f=a\x+\x^{3q-2}\in\Bbb F_{q^2}[\x]$, where $a\in\Bbb F_{q^2}^*$. Then $f$ is a PP of $\Bbb F_{q^2}$ if and only if one of the following occurs.
\begin{itemize}
  \item [(i)] $q=2^{2k+1}$ and $a^{\frac{q+1}3}$ is a primitive $3$rd root of unity. 
  \item [(ii)] $q=5$ and $a^2$ is a root of $(\x+1)(\x+2)(\x-2)(\x^2-\x+1)$.
  \item [(iii)] $q=2^3$ and $a^3$ is a root of $\x^3+\x+1$.
  \item [(iv)] $q=11$ and $a^4$ is a root of $(\x-5)(\x+2)(\x^2-\x+1)$.
  \item [(v)] $q=17$ and $a^6=4,5$.
  \item [(vi)] $q=23$ and $a^8=-1$.
  \item [(vii)] $q=29$ and $a^{10}=-3$.
\end{itemize}  
\end{thm}  

\noindent{\bf Note.} Shortly after the submission of the first vision of this paper, we were informed by M. Zieve of two very recent papers \cite{TZHL} by Tu, Zeng, Hu and Li and \cite{Zie13} by Zieve in the arXiv. We remark that under the assumption that $a$ is a $(q+1)$st root of unity, Theorem~\ref{T1.1} (i), which is the only relevant case under this assumption, follows from \cite[Corollary~5.3]{Zie13}. If, in addition, $q$ is assumed to be even, then \cite[Theorem~1]{TZHL} also gives sufficiency part of Theorem~\ref{T1.1} (i). 


\section{Preliminaries}

Let $f=a\x+\x^{3q-2}\in\Bbb F_{q^2}[\x]$, where $a\in\Bbb F_{q^2}^*$, and let $0\le\alpha,\beta\le q-1$. We have 
\begin{equation}\label{2.1}
\begin{split}
\sum_{x\in\Bbb F_{q^2}}f(x)^{\alpha+\beta q}
=\;&\sum_{x\in\Bbb F_{q^2}^*}(ax+x^{3q-2})^\alpha(a^qx^q+x^{3-2q})^\beta\cr
=\;&\sum_{x\in\Bbb F_{q^2}^*}\sum_{i,j}\binom\alpha i(ax)^{\alpha-i}x^{(3q-2)i}\binom\beta j(a^qx^q)^{\beta-j}x^{(3-2q)j}\cr
\end{split}
\end{equation}
\[
\kern 2.5cm =a^{\alpha+\beta q}\sum_{i,j}\binom\alpha i\binom\beta ja^{-i-jq}\sum_{x\in\Bbb F_{q^2}^*}x^{\alpha+\beta q+3(q-1)(i-j)}.
\]
The inner sum is $0$ unless $\alpha+\beta q\equiv 1\pmod{q-1}$, i.e., $\alpha+\beta=q-1$.

Assume $0\le \alpha\le q-1$ and $\beta=q-1-\alpha$ in \eqref{2.1}. We have
\begin{equation}\label{2.2}
\begin{split}
&\sum_{x\in\Bbb F_{q^2}}f(x)^{\alpha+(q-1-\alpha)q}\cr
=\;&a^{(\alpha+1)(1-q)}\sum_{i,j}\binom\alpha i\binom{q-1-\alpha}ja^{-i-jq}\sum_{x\in\Bbb F_{q^2}^*}x^{(q-1)[-\alpha-1+3(i-j)]}\cr
=\;&-a^{(\alpha+1)(1-q)}\sum_{-\alpha-1+3(i-j)\equiv 0\, (\text{mod}\, q+1)}\binom\alpha i\binom{q-1-\alpha}ja^{-i-jq}.
\end{split}
\end{equation}
As $i$ runs over the interval $[0,\alpha]$ and $j$ over the interval $[0,q-1-\alpha]$, the range of $-\alpha-1+3(i-j)$ is 
\begin{equation}\label{2.3}
I_\alpha:=[2\alpha+2-3q,\ \alpha-1].
\end{equation}
Thus we have 
\begin{equation}\label{2.4}
\sum_{x\in\Bbb F_{q^2}}f(x)^{\alpha+(q-1-\alpha)q}=-a^{(\alpha+1)(1-q)}S_q(\alpha,a),
\end{equation}
where
\begin{equation}\label{2.4a}
S_q(\alpha,a)= \sum_{-\alpha-1+3(i-j)\in I_\alpha\cap(q+1)\Bbb Z}\binom\alpha i\binom{q-1-\alpha}ja^{-i-jq}.
\end{equation}
By Hermite's criterion, $f$ is a PP of $\Bbb F_{q^2}$ if and only if  $0$ is the only root of $f$ and
\begin{equation}\label{2.5}
S_q(\alpha,a)=0\quad \text{for all}\ 0\le \alpha\le q-1.
\end{equation}

\begin{lem}\label{L2.1}
If $f$ is a PP of $\Bbb F_{q^2}$, then $q+1\equiv 0\pmod 3$.
\end{lem}

\begin{proof}
Assume $f$ is a PP of $\Bbb F_{q^2}$ and assume $q\ge 3$. If $q=3$, the only multiple of $q+1$ in $I_0=[2-3q,-1]$ is $-(q+1)$. By \eqref{2.5} (with $\alpha=0$), we have
\[
0=\sum_{-1-3j=-(q+1)}\binom{q-1}ja^{-jq}=-a^{-q},
\]
which is a contradiction. Now assume $q\ge 4$. The multiples of $q+1$ in $I_0=[2-3q,-1]$ are $-2(q+1)$ and $-(q+1)$. By \eqref{2.5} (with $\alpha=0$),
\begin{equation}\label{2.6}
0=\sum_{-1-3j=-2(q+1),-(q+1)}\binom{q-1}ja^{-jq}=\binom{q-1}{\frac{2q+1}3}^*a^{-\frac{2q+1}3q}+\binom{q-1}{\frac q3}^*a^{-\frac q3},
\end{equation}
where
\[
\binom mn^*=
\begin{cases}
\displaystyle\binom mn&\text{if}\ n\in\Bbb N,\vspace{2mm}\cr
0&\text{otherwise}.
\end{cases}
\]
If, to the contrary, $q+1\not\equiv 0\pmod 3$, then exactly one of $\binom{q-1}{\frac{2q+1}3}^*$ and $\binom{q-1}{\frac q3}^*$ is nonzero, and hence \eqref{2.6} cannot hold.
\end{proof}   

\begin{rmk}\label{R2.2}\rm
Assume $q+1\equiv 0\pmod 3$.
\begin{itemize}
  \item [(i)] $0$ is the only root of $f$ in $\Bbb F_{q^2}$ if and only if $a^{\frac{q+1}3}\ne 1$. (Note that $(-1)^{\frac{q+1}3}=1$.)
  \item [(ii)] For $b\in\Bbb F_{q^2}^*$, we have $f(b\x)=b^{3q-2}(b^{3(1-q)}a\x+\x^{3q-2})$. Thus if $f$ is a PP of $\Bbb F_{q^2}$ for $a=a_0$, the same is true for $a=\epsilon a_0$, where 
$\epsilon\in\Bbb F_{q^2}$, $\epsilon^{\frac{q+1}3}=1$.
\end{itemize} 
\end{rmk}

\begin{lem}\label{L2.3}
Assume $q+1\equiv 0\pmod 3$, $\alpha>0$, $\alpha+1\equiv 0\pmod 3$, $q\ge 2\alpha+4$. Then 
\begin{equation}\label{2.7}
S_q(\alpha,a)=(-a)^{\frac{q+1}3 q}\sum_{i=0}^\alpha(-1)^i\binom\alpha i\biggl[\binom{i+\frac{2\alpha-1}3}\alpha v^{3i}+\binom{i+\frac{2\alpha}3}\alpha v^{3i+1}+\binom{i+\frac{2\alpha+1}3}\alpha v^{3i+2}\biggr],
\end{equation}
where $S_q(\alpha,a)$ is defined in \eqref{2.4a} and $v=a^{-\frac{q+1}3}$.
\end{lem}

\begin{proof}
Since $\alpha>0$ and $q\ge 2\alpha+4$, the multiples of $q+1$ in $I_\alpha=[2\alpha+2-3q,\,2\alpha-1]$ are $-2(q+1),-(q+1),0$. We have
\[
\begin{split}
&\text{LHS of \eqref{2.7}}\cr
=\;&\sum_{-\alpha-1+3(i-j)=-2(q+1),-(q+1),0}\binom\alpha i\binom{q-1-\alpha} ja^{-i-jq}\cr
=\;&\sum_{i=0}^\alpha\binom\alpha i\sum_{l=0}^2\binom{-1-\alpha}{\frac 13(l(q+1)-\alpha-1)+i} a^{-i-[\frac13(l(q-1)-\alpha-1)+i]q}\cr
=\;&\sum_{i=0}^\alpha\binom\alpha i\sum_{l=0}^2(-1)^{\frac13(l(q+1)-\alpha-1)+i}\binom{\frac13(l(q+1)-\alpha-1)+i+\alpha}\alpha a^{\frac{\alpha+1}3q-\frac{q+1}3(l+3i)}\cr
&\kern5.2cm (\binom{-m}n=(-1)^n\binom{n+m-1}{m-1}\ \text{for}\ m,n\in\Bbb N)\cr
=\;&(-a)^{\frac{q+1}3 q}\sum_{i=0}^\alpha(-1)^i\binom\alpha i\sum_{l=0}^2\binom{i+\frac{2\alpha-1+l}3}\alpha v^{3i+l}.
\end{split}
\]
\end{proof}


\section{Proof of Theorem~\ref{T1.1}}

\begin{lem}\label{L3.1}
Assume $q+1\equiv 0\pmod 3$ and $y:=a^{\frac{q+1}3}$ is a primitive $3$rd root of unity. Then for $1\le s\le q^2-2$, 
\begin{equation}\label{3.1}
\sum_{x\in\Bbb F_{q^2}}f(x)^s=
\begin{cases}
\displaystyle a^{-\frac 16(q+1)(3q-2)}(1+y)&\text{if $q$ is odd and}\ \displaystyle  s=\frac{q^2-1}2,\vspace{2mm}\cr
0&\text{otherwise}.
\end{cases}
\end{equation}
\end{lem}

\begin{proof}
By the observation after \eqref{2.1}, we only have to consider $s=\alpha+(q-1-\alpha)q$, $0\le \alpha\le q-1$. By \eqref{2.4}, we may further assume $\alpha+1\equiv 0\pmod 3$.
There is nothing to prove if $q=2$. Thus we assume $q>3$. 

\medskip
$1^\circ$ We claim that $I_\alpha=[2\alpha+2-3q,\, 2\alpha-1]$ contains exactly three consecutive multiples of $q+1$ unless $\alpha=\frac{q-1}2$ ($q$ odd); in the latter case, $I_\alpha=[-2q+1,\, q-2]$ contains exactly two multiples of $q+1$ which are $-(q+1)$ and $0$.

The length of $I_\alpha$ is $2\alpha-1-(2\alpha+2-3q)=3(q-1)$. Since $2q<3(q-1)<3(q+1)$, there are at least $2$ and at most $3$ multiples of $q+1$ in $I_\alpha$. Let $k(q+1),(k+1)(q+1)\in I_\alpha$. If $I_\alpha=[2\alpha+2-3q,\,2\alpha-1]$ contains only two multiples of $q+1$, then 
\begin{equation}\label{3.2}
\begin{cases}
2\alpha+2-3q\ge k(q+1)-q,\cr
2\alpha-1\le (k+1)(q+1)+q.
\end{cases}
\end{equation}
(See Figure~\ref{F1}.)
\vspace{5mm}
\begin{figure}[h]
\[
\beginpicture
\setcoordinatesystem units <3mm,3mm> point at 0 0

\arrow <4pt> [0.3, 0.67] from -1 -2 to -13 -2
\arrow <4pt> [0.3, 0.67] from 1 -2 to 13 -2

\setlinear
\plot -16 0  16 0 /
\plot -14 0  -14 0.3 / 
\plot -6 0  -6 0.3 /
\plot 6 0  6 0.3 /
\plot 14 0  14 0.3 /
\plot -13 0  -13 -3 /
\plot 13 0  13 -3 /

\put {$I_\alpha$} at 0 -2
\put {$\scriptstyle k(q+1)-q$} [b] at -14 0.6
\put {$\scriptstyle k(q+1)$} [b] at -6 0.6 
\put {$\scriptstyle (k+1)(q+1)$} [b] at 6 0.6
\put {$\scriptstyle (k+1)(q+1)+q$} [b] at 14 0.6 
\endpicture
\]
\caption{Proof of \eqref{3.2}}\label{F1}
\end{figure}
\vspace{5mm}

\noindent
Since $2\alpha+2-3q,\; 2\alpha-1\equiv 0\pmod 3$, we have
\begin{equation}\label{3.3}
\begin{cases}
2\alpha+2-3q\ge k(q+1)-q+2,\cr
2\alpha-1\le (k+1)(q+1)+q-2.
\end{cases}
\end{equation}
Taking the difference of the two inequalities in \eqref{3.3}, we conclude that 
\[
\begin{cases}
2\alpha+2-3q = k(q+1)-q+2,\cr
2\alpha-1 = (k+1)(q+1)+q-2.
\end{cases}
\]
Since $0\le \alpha\le q-1$, we must have $k=-1$. Thus $\alpha=\frac{q-1}2$.

\medskip
$2^\circ$
First assume $\alpha\ne\frac{q-1}2$. Let $I_\alpha\cap (q+1)\Bbb Z=(q+1)L$, where $L$ is a set of three consecutive integers. Then we have
\begin{equation}\label{3.4}
\begin{split}
S_q(\alpha, a)=\;&\sum_{-\alpha-1+3(i-j)\in I_\alpha\cap(q+1)\Bbb Z}\binom\alpha i\binom{q-1-\alpha}j a^{-i-jq}\cr
=\;&\sum_{l\in L}\;\sum_{-\alpha-1+3(i-j)=l(q+1)}\binom\alpha i\binom{q-1-\alpha}j a^{-i-jq}\cr
=\;&\sum_{l\in L}\;\sum_{i-j=\frac13[\alpha+1+l(q+1)]}\binom\alpha i\binom{q-1-\alpha}j a^{-i+j}\qquad\text{(since $a^{q+1}=1$)}\cr
=\;&\sum_{l\in L}a^{-\frac 13[\alpha+1+l(q+1)]}\sum_{i-j=\frac13[\alpha+1+l(q+1)]}\binom\alpha {\alpha-i}\binom{q-1-\alpha}j\cr
=\;&a^{-\frac 13(\alpha+1)}\sum_{l\in L}y^{-l}\sum_{\alpha-i+j=\frac13[2\alpha-1-l(q+1)]}\binom\alpha {\alpha-i}\binom{q-1-\alpha}j\cr
=\;&a^{-\frac 13(\alpha+1)}\sum_{l\in L}y^{-l}\binom{q-1}{\frac 13[2\alpha-1-l(q+1)]}\cr
=\;&-a^{-\frac 13(\alpha+1)}\sum_{l\in L}y^{-l}\cr
=\;&0.
\end{split}
\end{equation}
In the next-to-last step, we have $\binom{q-1}{\frac 13[2\alpha-1-l(q+1)]}=(-1)^{\frac 13[2\alpha-1-l(q+1)]}=-1$ since $0\le \frac 13[2\alpha-1-l(q+1)]\le \frac 13[2\alpha-1-(2\alpha+2-3q)]=q-1$. Combining \eqref{2.4} and \eqref{3.4} gives $\sum_{x\in\Bbb F_{q^2}}f(x)^s=0$.
 
\medskip
$3^\circ$ Now assume $\alpha=\frac{q-1}2$. By the calculation in \eqref{3.4}, we have
\begin{equation}\label{3.5} 
S_q(\alpha, a)=-a^{-\frac 13(\alpha+1)}\sum_{l=-1,0}y^{-l}=-a^{-\frac 16(q+1)}(1+y).
\end{equation}
Combining \eqref{2.4} and \eqref{3.5} gives
\[
\sum_{x\in\Bbb F_{q^2}}f(x)^s=a^{\frac{q+1}2(1-q)-\frac{q+1}6}(1+y)=a^{-\frac 16(q+1)(3q-2)}(1+y).
\]
\end{proof}

\begin{proof}[Proof of Theorem~\ref{T1.1}] 
($\Leftarrow$) (ii) -- (vii) are sporadic cases with small $q$. Using a computer, it is easy to verify that $f$ is a PP in each of these cases. Now assume (i). By Lemma~\ref{L3.1}, we have $\sum_{x\in\Bbb F_{q^2}}f(x)^s=0$ for all $1\le s\le q^2-2$. Also note from Remark~\ref{R2.2} (i) that $0$ is the only root of $f$ in $\Bbb F_{q^2}$. Thus $f$ is a PP of $\Bbb F_{q^2}$.

\medskip
($\Rightarrow$) Assume that $f$ is a PP of $\Bbb F_{q^2}$. By Lemma~\ref{L2.1}, we have $q+1\equiv 0\pmod 3$. Let $y=a^{\frac{q+1}3}$. If $y$ is a primitive $3$rd root of unity, by Lemma~\ref{L3.1}, $q$ must be even, i.e., $q=2^{2k+1}$, and we have case (i). Now assume that $y^2+y+1\ne 0$. We show that one of the cases (ii) -- (vii) occurs.

The sum in the right side of \eqref{2.7} is a polynomial in $y^{-1}$ ($=v$) which can be easily computed for small $\alpha$ with computer assistance. For $\alpha=2,5,8,11,14$, we find that 
\begin{equation}\label{3.6}
S_q(\alpha,a)=(-a)^{\frac{\alpha+1}3q}y^{-3\alpha-2}(y^2+y+1)
\begin{cases}
3^{-2}g_2(y)&\text{if}\ \alpha=2,\ q\ge 8,\vspace{1mm}\cr
3^{-6}g_5(y)&\text{if}\ \alpha=5,\ q\ge 14,\vspace{1mm} \cr
3^{-10}g_8(y)&\text{if}\ \alpha=8,\ q\ge 20,\vspace{1mm} \cr
3^{-15}g_{11}(y)&\text{if}\ \alpha=11,\ q\ge 26,\vspace{1mm} \cr
3^{-19}g_{14}(y)&\text{if}\ \alpha=14,\ q\ge 32,
\end{cases}
\end{equation}
where
\[
\begin{split}
g_2(\x)=\,&2 \x^5+3 \x^4-23 \x^3-8 \x^2-9 \x+44, \cr
g_5(\x)=\,&-14 \x^{14}-8 \x^{13}+22 \x^{12}-469 \x^{11}-1093 \x^{10}+8852 \x^9+6801 \x^8+10527 \x^7\cr
&-61068 \x^6-18619 \x^5-25033 \x^4+120197 \x^3+13516 \x^2+16822 \x-71162, \cr
g_8(\x)=\,& 130 \x^{23}+57 \x^{22}-187 \x^{21}+4082 \x^{20}+3585 \x^{19}-7667 \x^{18}+156234 \x^{17}\cr
&+453573 \x^{16}-3916551 \x^{15}-4144622 \x^{14}-7594467 \x^{13}+48939959 \x^{12}\cr
&+25221008 \x^{11}+39342423 \x^{10}-213366911 \x^9-61811112 \x^8-88032825 \x^7\cr
&+422650317 \x^6+66303028 \x^5+88882095 \x^4-389019163 \x^3-25886212 \x^2\cr
&-33211905 \x+135094180, \cr
g_{11}(\x)=\,& -3952 \x^{32}-1522 \x^{31}+5474 \x^{30}-139802 \x^{29}-89324 \x^{28}+229126 \x^{27}\cr
&-3943602 \x^{26}-4392909 \x^{25}+8336511 \x^{24}-180820302 \x^{23}-605825169 \x^{22}\cr 
&+5521784781 \x^{21}+7111655988 \x^{20}+14607372831 \x^{19}-101269369227 \x^{18}\cr
&-69095625624 \x^{17}-119477261853 \x^{16}+705650100129 \x^{15}+303870716124 \x^{14}\cr
&+475920749355 \x^{13}-2503382174319 \x^{12}-706243777836 \x^{11}-1034492806725 \x^{10}\cr
&+4972469163636 \x^9+898579001889 \x^8+1253008322595 \x^7-5598768742164 \x^6\cr
&-591556509206 \x^5-794043854630 \x^4+3339003167188 \x^3+157572058982 \x^2\cr
&+205140400010 \x-819352075360, \cr
g_{14}(\x)=\,& 41800 \x^{41}+14895 \x^{40}-56695 \x^{39}+1691000 \x^{38}+905631 \x^{37}-2596631 \x^{36}\cr
&+47250150 \x^{35}+37894401 \x^{34}-85144551 \x^{33}+1395800990 \x^{32}+1826164521 \x^{31}\cr
&-3221965511 \x^{30}+75566097190 \x^{29}+281332431561 \x^{28}-2683745985685 \x^{27}\cr
&-3976231919076 \x^{26}-8901790877799 \x^{25}+65232090577890 \x^{24}+53701334712609 \x^{23}\cr
&+100487514543597 \x^{22}-632854611825486 \x^{21}-347885978019711 \x^{20}\cr
&-586551837541203 \x^{19}+3307822221633594 \x^{18}+1283881108529889 \x^{17}\cr
&+2015859062567817 \x^{16}-10419893389315746 \x^{15}-2892546806289271 \x^{14}\cr
&-4307726185011783 \x^{13}+20728564105915330 \x^{12}+4054215382726378 \x^{11}\cr
&+5793391583605092 \x^{10}-26245535590106350 \x^9-3451745974770042 \x^8\cr
&-4770402189292728 \x^7+20520594631893930 \x^6+1634454816505198 \x^5\cr
&+2197118421394272 \x^4-9034762128135730 \x^3-330180086243950 \x^2\cr
&-433563200685120 \x+1713531735146800.
\end{split}
\]
Combining \eqref{2.5} and \eqref{3.6}, we have
\begin{equation}\label{3.7}
\begin{cases}
g_2(y)=0&\text{if}\ q\ge 8,\cr
g_5(y)=0&\text{if}\ q\ge 14,\cr
g_8(y)=0&\text{if}\ q\ge 20,\cr
g_{11}(y)=0&\text{if}\ q\ge 26,\cr
g_{14}(y)=0&\text{if}\ q\ge 32.
\end{cases}
\end{equation}
When $q<14$, a quick computer search produces cases (ii) -- (iv). So we assume $q\ge 14$. Next we compute the resultant of $g_2$ and $g_5$:
\[
R(g_2,g_5)=2^5\cdot 3^{35}\cdot 17^2\cdot 23\cdot 29 \cdot 103\cdot 16069.
\]
By \eqref{3.7} and the fact that $q+1\equiv 0\pmod 3$, we must have $p=\text{char}\,\Bbb F_q\in\{2,17,23,29\}$.

When $p=2$, we have $q\ge 32$. In this case, $\text{gcd}(g_2,g_5,g_8)=\x$, which is a contradiction to \eqref{3.7}.

When $p=17$, we find that $\text{gcd}(g_2,g_5,g_8)=1$. By \eqref{3.7}, we must have $q=17$. A computer search results in case (v).

When $p=23$, $\text{gcd}(g_2,g_5,g_8)=\x+1$, and $g_{11}(-1)=12\ne 0$. By \eqref{3.7}, we must have $q=23$. A computer search results in case (vi).

When $p=29$, $\text{gcd}(g_2,g_5,g_8)=\x+10$, and $g_{11}(-10)=0$, $g_{14}(-10)=2\ne 0$. By \eqref{3.7}, we must have $q=29$. A computer search results in case (vii).
\end{proof}


\end{document}